\documentclass[12pt]{amsart}
\usepackage{amsmath,amssymb,mathtools,amsthm,textcomp,mathscinet}
\usepackage{amsfonts,graphicx,enumerate,subcaption}
\usepackage{color,tikz,float}
\usepackage[mathscr]{eucal}
\usepackage[foot]{amsaddr}
\theoremstyle{definition}
\usepackage[american]{babel}
\usepackage{csquotes}
\usepackage{placeins}

\def \C{{{\rm I{\!\!\!}\rm C}}}
\numberwithin{equation}{section}
\newtheorem{theorem}{\bf Theorem}[section]
\newtheorem{remark}{\bf Remark}[section]

\newtheorem{lemma}{Lemma}[section]
\newtheorem{corollary}{Corollary}[section]
\newtheorem{example}{Example}[section]
\newtheorem{definition}{Definition}[section]

\usepackage{chngcntr}
\counterwithout{algorithm}{section}

\newtheoremstyle
    {remarkstyle}
    {}
    {11pt}
    {}
    {}
    {\bfseries}
    {:}
    {     }
    {\thmname{#1} \thmnumber{#2} }

\theoremstyle{remarkstyle}

\begin{document}
\title{Non-homogeneous space-time fractional Poisson processes}
\author{\small A. Maheshwari and P. Vellaisamy }
\address{\small Department of Mathematics,
Indian Institute of Technology Bombay, Powai, Mumbai 400076, INDIA.}
\email{aditya@math.iitb.ac.in, pv@math.iitb.ac.in}
\date{\today}
\begin{abstract}
 The space-time fractional Poisson process (STFPP), defined by Orsingher and Poilto in \cite{sfpp}, is a generalization of the time fractional Poisson process (TFPP) and the space fractional Poisson process (SFPP). We study the fractional generalization of the non-homogeneous Poisson process and call it the non-homogeneous space-time fractional Poisson process (NSTFPP). We compute their {\it pmf} and generating function and investigate the associated differential equation. The limit theorems and the law of iterated logarithm for the NSTFPP process are studied. We study the distributional properties, the asymptotic expansion of the correlation function of the non-homogeneous time fractional Poisson process (NTFPP) and subsequently investigate the long-range dependence (LRD) property of a special  NTFPP. We investigate the limit theorem and the LRD property for the fractional non-homogeneous Poisson process (FNPP), studied by Leonenko et. al. (2016). Finally, we present some simulated sample paths of the NSTFPP process.
\end{abstract}

\maketitle
\section{Introduction}
The study of Poisson process and its applications has attracted lot of attention among the researchers, scientists and engineers. Recent times witness a growing interest in its fractional version called as the fractional Poisson process (FPP). The non-homogeneous Poisson process (NPP) can be considered as the Poisson process where the time variable is replaced by rate function $\Lambda(t),~t\geq0$, that is, $\widetilde{N}(t)=N(\Lambda(t),1)$. In this paper, we study a space-time fractional generalization of the NPP. \\\\
 A fractional version of the NPP is recently studied by Leonenko {\it et al.} \cite{fnhpp} defined by time-changing the NPP $\{\widetilde{N}(t)\}_{t\geq0}$ by an inverse $\beta$-stable subordinator. In their paper, they also pointed out an another way of defining the fractional version of the NPP (see \cite[Section 6]{fnhpp}) which is defined by replacing the the time variable of the FPP $\{N_\beta(t)\}_{t\geq 0}$ by the rate function $\Lambda(t),~t\geq0$. It is important to underline that these two processes are different. We expound on the latter definition and study the approach in a general sense in this paper.  \\\\
  In the literature there are two versions of the FPP, namely the time fractional Poisson process (TFPP) (see \cite{lask}) and the space fractional Poisson process (SFPP) (see \cite{sfpp}), which are defined by time-changing the Poisson process by an inverse $\beta$-stable subordinator and a $\beta$-stable subordinator, respectively. These two processes form a special case for the space-time fractional Poisson process (STFPP) which is defined by Orsingher and Poilto (see \cite{sfpp}). We here study the space-time fractional generalization of the NPP called as the non-homogeneous space-time fractional Poisson process (NSTFPP). We also study the fractional non-homogeneous Poisson process (FNPP), investigated by Leonenko et. al. (see \cite{fnhpp}), which is defined by time-changing the NPP with the inverse $\beta$-stable subordinator.\\ 
  
  We give the {\it pmf} and the generating function for NSTFPP process.  The fractional differential equation for the generating function of the NSTFPP is derived. The limit theorem and the law of iterated logarithm for the NSTFPP process are derived. A representation for the generating function of the NSTFPP is also given. The asymptotic expansion for the correlation function of the non-homogeneous time fractional Poisson process (NTFPP) is discussed. The long-range dependence (LRD) property of a special case for the NTFPP is studied. A limit theorem and the LRD property of the FNPP process are established. Lastly, we present some simulated sample paths for some special NSTFPP process. \\\\
  The paper is organized as follows. In Section \ref{sec:Prelims}, we discuss some preliminaries and definitions which are required for the paper. In Section \ref{sec:NSTFPP}, we investigate the NSTFPP process. Simulations for some special cases for the NSTFPP process are presented in Section \ref{sec:Simulation}.
\section{Preliminaries} \label{sec:Prelims}
\noindent In this section we present some preliminary results which are required later in the paper. Let $\mathbb{Z}_{+}=\{0,1,\ldots,\}$ be the set of nonnegative integers. \\
\noindent   The Mittag-Leffler function $L_{\alpha}(z)$ is defined as (see \cite{Mittag-Leffler-original})
\begin{equation}\label{Mittag-Leffler-function}
L_{\alpha}(z)=\sum\limits_{k=0}^{\infty}\frac{z^{k}}{\Gamma(1+\beta k)},\,\,\,\alpha,z\in \C \text{ and } \Re(\alpha)>0.
\end{equation}
\subsection{Stable and inverse stable subordinator}
Let $\{D_\beta(t)\}_{t\geq0}$ be the $\beta$-stable subordinator (see \cite{appm}) with Laplace transform (LT)
\begin{equation}
	\mathbb{E}[e^{-sD_\beta(t)}]=e^{-ts^\beta}.
\end{equation}
The inverse $\beta$-stable subordinator (see \cite{bingham71,MeerStrak13}) is defined as the right-continuous inverse of the $\beta$-stable subordinator
\begin{equation}
	E_\beta(t)=\inf\{r>0:D(r)>t \},~~t\geq0.
 \end{equation}
 A stochastic process $\{X(t)\}_{t\geq0}$
 is self-similar (see \cite{appm}) with Hurst index $H > 0$ if 
 \begin{equation*}
  X(ct)\stackrel{d}{=}c^HX(t),
 \end{equation*} have the same finite-dimensional distributions for all $c\geq0.$
 It is well known that the $\beta$-stable subordinator is  self-similar with Hurst index $1/\beta$, that is,
 \begin{equation}\label{stable-ss}
 	D_\beta(ct)\stackrel{d}{=}c^{1/\beta}D_\beta(t),~c>0.
 \end{equation}
 Also, it can be seen that (see e.g. \cite{fnbpfp,MeerSche2004}) the inverse $\beta$-stable subordinator is self-similar with Hurst index $\beta$, that is
 \begin{equation}\label{istable-ss}
 E_\beta(ct)\stackrel{d}{=}c^{\beta}E_\beta(t),~c>0.
 \end{equation}
 \subsection{The LRD property}
 \noindent There are several definitions in the literature for the LRD and the short-range dependence (SRD) property of a stochastic process. We now present our definition (see \cite{ovi-lrd,lrd2016}) which will be used in this paper.
 \begin{definition}\label{LRD-definition}
 	Let $s>0$ be fixed and $t>s$. Suppose a stochastic process $\{X(t)\}_{t\geq0}$ has the correlation function Corr$[X(s),X(t)]$ that satisfies
 	\begin{equation}\label{LRD-defn1}
 	c_1(s)t^{-d}\leq\text{Corr}[X(s),X(t)]\leq c_2(s)t^{-d},
 	\end{equation}
 	for large $t$, $d>0$, $c_1(s)>0$ and $c_2(s)>0$. In other words, 
 	\begin{equation}\label{LRD-defn2}
 	\lim\limits_{t\to\infty}\frac{\text{Corr}[X(s),X(t)]}{t^{-d}}=c(s),
 	\end{equation}for some $c(s)>0$ and $d>0.$  We say $\{X(t)\}_{t\geq0}$ has the LRD property if $d\in(0,1)$  and has the SRD property if $d\in(1,2)$.
 \end{definition}
 Note that \eqref{LRD-defn1} and \eqref{LRD-defn2} are equivalent and imply that Corr$[X(s),X(t)]$ behaves like $t^{-d}$, for large $t$.

\section{Non-homogeneous space-time fractional Poisson process}\label{sec:NSTFPP}
 The space-time fractional Poisson process (STFPP) is defined in Orsingher and Polito (see \cite[Remark 2.4]{sfpp}). The STFPP generalizes both the time and the space fractional Poisson process.
 \begin{definition}[Space-time fractional Poisson process]
 	\noindent Let $0<\alpha,\beta\leq1$. The space-time fractional Poisson process (STFPP) $\{N^\alpha_{\beta}(t,\lambda)\}_{t\geq0}$, which is a generalization of the Poisson process $\{N(t,\lambda)\}_{t\geq0}$, is defined to be a stochastic process for which $p^\alpha_{_\beta}(n|t,\lambda)=\mathbb{P}[N^\alpha_{\beta}(t,\lambda)=n]$ satisfies (see \cite{sfpp})
 	
 	\begin{align}\label{def:stfpp}
 		D^\beta_t p^\alpha_{_{\beta}}(n|t,\lambda) &= -\lambda^\alpha(1-B_n)^\alpha p^\alpha_{_{\beta}}(n|t,\lambda),~\text{for }n\geq1,\\
 		D^{\beta}_{t}p^\alpha_{_{\beta}}(0|t,\lambda) &= -\lambda p^\alpha_{_{\beta}}(0|t,\lambda)\nonumber
 	\end{align}
 	$\text{with  }p^\alpha_{_{\beta}}(n|0,\lambda)=1\text{ if }n=0 \text{ and is zero if }n\geq1.$  Here, $D^{\beta}_{t}$ denotes the fractional derivative in the Caputo sense and is defined as
 	\begin{equation}\label{cd}
 		D_{t}^{\beta}f(t)= \begin{cases} 
 			\hfill \dfrac{1}{\Gamma(1-\beta)}\displaystyle\int\limits_{0}^{t}\frac{f'(s)}{(t-s)^{\beta}}ds, \hfill    &0<\beta<1 , \\ 
 			f'(t), \,\,\,\,\,\,\,\,\,\,\, \beta=1,&
 		\end{cases}
 	\end{equation}
 	where $f'$ denotes the derivative of $f.$  
 \end{definition}
 \noindent The {\it pmf} $p^\alpha_{_{\beta}}(n|t,\lambda)$ of the STFPP is given by (see \cite[eq. (2.29)]{sfpp}) 
 \begin{equation}\label{pmf:stfppd}
 	p^\alpha_{_{\beta}}(n|t,\lambda)=\frac{(-1)^n}{n!} \sum_{k=0}^\infty \frac{(-\lambda^\alpha t^\beta)^k}{
 		\Gamma(\beta k+1)} \frac{\Gamma(\alpha k+1)}{\Gamma(\alpha k + 1- n)}, ~~n \geq 0, \: \alpha,\beta \in
 	(0,1].
 \end{equation}	
 The probability generating function ({\it pgf}) of the STFPP is given by (see \cite[eq. (2.28)]{sfpp})
 \begin{align}\label{pgf-stfpp}
 	G^{\alpha,\beta}_\lambda(s,t) = L_\beta(-\lambda^\alpha t^\beta (1-s)^\alpha),~|s| \leq 1.	\end{align}
 The fractional differential equation governing the
 {\it pgf} of the STFPP is given by (see \cite[eq. (2.27)]{sfpp})
 \begin{align*}
 	D_t^\beta G^{\alpha,\beta}_\lambda(s,t) = -\lambda^\alpha(1-u)^\alpha
 	G^{\alpha,\beta}_\lambda(s,t) , ~~ |u| \leq 1, \: \alpha,\beta \in (0,1]
 \end{align*}
 with $G^{\alpha,\beta}_\lambda(s,0) = 1.$ 	
 
 A different characterization of the STFPP is to subordinate the Poisson process $\{N(t)\}_{t\geq0}$ by an independent $\alpha$-stable subordinator $\{D_\alpha(t)\}_{t\geq0}$ and then by the inverse $\beta$-stable subordinator $\{E_\beta(t)\}_{t\geq0}$
 \begin{equation}\label{stfpp-subord}
 	N_\beta^\alpha(t,\lambda)\stackrel{d}{=}N(D_\alpha(E_\beta(t)),\lambda),~t\geq0.
 \end{equation}
 To see this, let us compute the {\it pgf} of the $N(D_\alpha(E_\beta(t)),\lambda)$
 \begin{align*}
 	\mathbb{E}\left[s^{N(D_\alpha(E_\beta(t)),\lambda)}\right]&=\mathbb{E}\left[\mathbb{E}\left[s^{N(D_\alpha(E_\beta(t)),\lambda)}\big|E_\beta(t)\right]\right]=\mathbb{E}\left[e^{-(1-s)^\alpha\lambda^\alpha E_\beta(t)}\right]=L_\beta(-\lambda^\alpha(1-s)^{\alpha}t^\beta),
 \end{align*}
 which coincides with the {\it pgf} of the STFPP given in \eqref{pgf-stfpp}.
 \subsection*{Special cases}
 The STFPP reduces to the time fractional Poisson process (TFPP) (see \cite{lask}) and space fractional Poisson process (SFPP) (see \cite{sfpp}) when taking $\alpha=1$ and $\beta=1$, respectively in \eqref{def:stfpp}. 
 \begin{remark}
 For the comparison of the stochastic representation of the STFPP, given in \eqref{stfpp-subord}, with that of the TFPP (see \cite{mnv}) and the SFPP (see \cite{sfpp}), we assume $E_1(t)=D_1(t)=t$, {\it a.s.}
  \end{remark}
\noindent We now define the non-homogeneous version of the STFPP. 
\begin{definition}\label{def:NSTFPP}
	The non-homogeneous space-time fractional Poisson process (NSTFPP) is defined as 
	\begin{equation}\label{def:ntsfpp}
	W^\alpha_\beta(t)=N^\alpha_\beta(\Lambda(t),1),~t\geq0,
	\end{equation}
	where $\{N^\alpha_\beta(t,\lambda)\}_{t\geq0}$ is the STFPP and $\Lambda(t)=\int_{0}^{t}\lambda(u)du$ is the rate function with intensity function $\lambda(u),u\geq0.$
\end{definition}

 In view of \eqref{stfpp-subord}, the NSTFPP $\{W^\alpha_\beta(t)\}_{t\geq0}$ can also be seen as 
\begin{equation}\label{nstfpp-subord}
W^\alpha_\beta(t)\stackrel{d}{=}N(D_\alpha(E_\beta(\Lambda(t))),1),~t\geq0.
\end{equation}
\begin{remark}
	\noindent Note that when we take $\lambda(u)=\lambda^{\alpha/\beta},u\geq0\Rightarrow\Lambda(t)=\lambda^{\alpha/\beta }t$, the NSTFPP reduces to the STFPP. To see this, observe from \eqref{nstfpp-subord} that
\begin{align*}
W^\alpha_\beta(t)&\stackrel{d}{=}N(D_\alpha(E_\beta(\lambda^{\alpha/\beta}t)),1)\\
&\stackrel{d}{=}N(D_\alpha(\lambda^{\alpha}E_\beta(t)),1)~~~\text{(using \eqref{istable-ss}})\\
&\stackrel{d}{=}N(\lambda D_\alpha(E_\beta(t)),1)~~~\text{(using \eqref{stable-ss}})\\
&\stackrel{d}{=}N( D_\alpha(E_\beta(t)),\lambda)\stackrel{d}{=}N_\beta^\alpha(t,\lambda),
\end{align*}
as stated.
\end{remark}
\noindent The {\it pmf } of the NSTFPP can be directly computed using the {\it pmf } of the STFPP given in \eqref{pmf:stfppd} by replacing $t$ by $\Lambda(t)$ and $\lambda$ by 1 and is given by
\begin{align}\label{nsfpp-pmf}
p^\alpha_{_\beta}(n|\Lambda(t))=\mathbb{P}[W^\alpha_\beta(t)=n]&=\frac{(-1)^n}{n!} \sum_{k=0}^\infty \frac{(-\Lambda^\beta(t))^k}{
	\Gamma(\beta k+1)} \frac{\Gamma(\alpha k+1)}{\Gamma(\alpha k + 1- n)}, ~~n \geq 0.
\end{align}
The NSTFPP process under consideration encapsulates a number of models which may used in applications. In the following examples, we try to list few of them. 
\begin{example}[Weibull distribution]
	The NPP process with Weibull intensity and rate function is given by
	\begin{equation*}
	\lambda(s)=\left(\dfrac{a}{b}\right)\left(\dfrac{s}{b}\right)^{a-1}~\text{and~ }\Lambda(t)=\left(\frac{t}{b}\right)^a,~a,b>0,
	\end{equation*} 
	is used to model the damage process. Damage process can be manifested by accumulation of damage viz. rusting, cracks etc., which eventually leads to failure (see \cite{Weibull-NHPP} and references therein). In \cite{Weibull-NHPP-2}, it is used to model failure times of repairable systems. Its fractional generalization studied in this paper may model this process and can be subject of future study. From \eqref{nstfpp-subord}, the corresponding NSTFPP $\{W^\alpha_\beta(t)\}_{t\geq0}$ is given by 
	\begin{equation*}
W^\alpha_\beta(t)\stackrel{d}{=}N\left(D_\alpha\left(E_\beta\left(t^a/b^a\right)\right),1\right),~a,b>0.
	\end{equation*}
\end{example}

\begin{example}[Gompertz-Makeham distribution]
	Software reliability models are often modeled using the NPP with Gompertz rate (see \cite{NHPP-Gompertz,NHPP-Gompertz-2,NHPP-hazard} and references therein). The Gompertz-Makeham distribution is a generalization of the Gompertz distribution. The NPP associated with Gompertz-Makeham intensity and rate is given by
	\begin{equation*}
	\lambda(s)=ab e^{b s}+\mu,  \text{~~and~~}	\Lambda(t)=a(e^{b t}-1)+\mu t,~~a,b,\mu>0.
	\end{equation*}
	Note that when $\mu=0$, the Gompertz-Makeham intensity (rate) function reduces to Gompertz intensity (rate) function. 
	Software reliability models based on NPP are used to estimate software reliability and also found useful predict software faults, release times, failure rates and etc. The present fractional generalization of this models may find application in this area. Using \eqref{nstfpp-subord}, the corresponding NSTFPP $\{W^\alpha_\beta(t)\}_{t\geq0}$ is given by 
	\begin{equation*}
	W^\alpha_\beta(t)\stackrel{d}{=}N\left(D_\alpha\left(E_\beta\left(a(e^{bt}-1)+\mu t\right)\right),1\right),~a,b,\mu>0.
	\end{equation*}
\end{example}
\begin{example}[Musa-Okumoto model]
	The NPP with Musa-Okumoto intensity can be used as a reliability model in software testing (see \cite{NHPP-hazard} and the references therein). Its intensity and rate function is given by
	\begin{equation*}
	\lambda(s)=\frac{ab}{1+bt},~~\text{and~~}\Lambda(t)=a\ln(1+bt),~a,b>0.
	\end{equation*}
	The fractional generalization of this model may be interest in reliability testing. From \eqref{nstfpp-subord}, the corresponding NSTFPP $\{W^\alpha_\beta(t)\}_{t\geq0}$ is given by 
	\begin{equation*}
	W^\alpha_\beta(t)\stackrel{d}{=}N\left(D_\alpha\left(E_\beta\left(a\ln(1+bt)\right)\right),1\right),~a,b>0.
	\end{equation*}
\end{example}
The following table gives some of the important distributions with the corresponding intensity and rate function.
\renewcommand{\arraystretch}{2.5}

\begin{table}[H]
	\begin{tabular}{ |c|  c|  c|}
		\hline
		Distribution name & Intensity function $\lambda(s)$ & Rate function  $\Lambda(t)$   \\ \hline
		Weibull distribution
		& $\dfrac{a}{b}\left(\dfrac{s}{b}\right)^{a-1},a,b>0$ & $\left(\dfrac{t}{b}\right)^a $ \\ \hline
		
		Gompertz-Makeham's distribution &$a be^{b s}+\mu, a, b,\mu>0
		
		$&$a\left(e^{b t}-1\right)+\mu t$\\ \hline
		
		Musa-Okumoto distribution &$\dfrac{ab}{1+bs}, a, b,\mu>0
		
		$&$a\ln(1+bt)$\\ \hline
	\end{tabular}\vspace*{0.5cm}
	\caption{Examples for intensity and rate functions.}\label{table-example}
\end{table}
We next present the limit theorem for the NSTFPP.
\begin{theorem}
	
	If $\Lambda(t)\to\infty,$ as $t\to\infty$. Then
	\begin{equation}
	\lim\limits_{t\to\infty}\frac{W^\alpha_\beta(t)}{\Lambda^{\beta/\alpha}(t)}= D_\alpha(E_\beta(1)), ~a.s.
	\end{equation}
	\begin{proof}
		Using \eqref{istable-ss}, we have that
		\begin{equation*}
		W^\alpha_\beta(t)=N(D_\alpha(E_{\beta}(\Lambda(t))),1)\stackrel{d}{=}N\left(D_\alpha(\Lambda^{\beta}(t)E_{\beta}(1)),1\right)\stackrel{d}{=}N\left(\Lambda^{\beta/\alpha}(t)D_\alpha(E_\beta(1)),1\right).
		\end{equation*}
		The law of large numbers for the Poisson process implies
		\begin{equation}\label{llnPoisson}
		\lim_{t\rightarrow\infty}\frac{N(t,1)}{t}=1,~{a.s.}
		\end{equation}
		Next, consider
		\begin{align*}
		\lim\limits_{t\to\infty}\frac{W^\alpha_\beta(t)}{\Lambda^{\beta/\alpha}(t)}&=\lim\limits_{t\to\infty}\frac{N^\alpha_\beta(\Lambda(t),1)}{\Lambda^{\beta/\alpha}(t)}=\lim\limits_{t\to\infty}\frac{N(D_\alpha(E_\beta(\Lambda(t))),1)}{\Lambda^{\beta/\alpha}(t)}\\
		&=\lim\limits_{t\to\infty}\frac{N(\Lambda^{\beta/\alpha}(t)D_\alpha(E_\beta(1)),1)}{\Lambda^{\beta/\alpha}(t)}\\
		&=\lim\limits_{t\to\infty}\frac{N(\Lambda^{\beta/\alpha}(t)D_\alpha(E_\beta(1)),1)}{\Lambda^{\beta/\alpha}(t)D_\alpha(E_\beta(1))}\frac{\Lambda^{\beta/\alpha}(t)D_\alpha(E_\beta(1))}{\Lambda^{\beta/\alpha}(t)}\\
		&=\lim\limits_{t\to\infty}\frac{\Lambda^{\beta/\alpha}(t)D_\alpha(E_\beta(1))}{\Lambda^{\beta/\alpha}(t)},~a.s.~~(\text{using }\eqref{llnPoisson})\\
		&=D_\alpha(E_\beta(1)), ~a.s.\qedhere
		\end{align*}
	\end{proof}
\end{theorem}
\noindent We now discuss the law of iterated logarithm (LIL) for the NSTFPP process. First, we have some result and definition which are required to prove the LIL for the NSTFPP.  \begin{definition}
	We call a function $l:(0,\infty)\rightarrow(0,\infty)$ regularly varying at 0+ with index $\alpha\in\mathbb{R}$ (see \cite{bertoin}) if 
	$$\lim_{x\rightarrow 0+}\frac{l(\lambda x)}{l(x)}=\lambda^\alpha,~\text{for}~ \lambda>0.$$
\end{definition}
\noindent   We first give the special case of LIL for the $\beta$-stable subordinator (for general case see Bertoin \cite[Chapter III, Theorem 14]{bertoin}).
\begin{lemma}
	Let $\{D_\beta(t)\}_{t\geq0}$ be a $\beta$-stable subordinator with $\mathbb{E}[e^{-sD_\beta(t)}]=e^{-ts^\beta}$, $\beta\in(0,1)$ and
	\begin{equation*}
	g(t)=\frac{\log\log t}{(t^{-1}\log\log t)^{1/\beta}},~(e<t).
	\end{equation*}
	Then
	\begin{equation}\label{LIL-sub}
	\liminf_{t\to\infty}\frac{D_\beta(t)}{g(t)}=\beta(1-\beta)^{(1-\beta)/\beta},~~a.s.
	\end{equation}
\end{lemma}

\begin{theorem}
	If $\Lambda(t)\to\infty$ as $t\to\infty$. Then 
	\begin{equation}\label{LIL-tcpp}
	\liminf_{t\rightarrow\infty}\frac{W^\alpha_\beta(t)}{g(t)}= \alpha\left(1-\alpha\right)^{(1-\alpha)/\alpha}E_\beta^{1/\alpha}(1)~~{a.s.},
	\end{equation}
	where 
	\begin{equation*}
	g(t)=\frac{\log\log \Lambda^\beta(t)}{((\Lambda^\beta(t))^{-1}\log\log \Lambda^\beta(t))^{1/\alpha}},~(\Lambda(t)>e^{1/\beta}).
	\end{equation*}
	\begin{proof}
		The law of large numbers for the Poisson process implies
		\begin{equation}\label{LIL-Poisson}
		\lim_{t\rightarrow\infty}\frac{N(t,1)}{t}=1,~{a.s.}
		\end{equation} Note that $D_\beta(t)\to\infty$, $a.s.$ as $t\to\infty$ (see  \cite{MeerSche2004}). Consider now,
		\begin{align*}
		\liminf_{t\rightarrow\infty}\frac{W^\alpha_\beta(t)}{g(t)}&=\liminf_{t\rightarrow\infty}\frac{N(D_{\alpha}(E_\beta(\Lambda(t))),1)}{g(t)}\\
		&=\liminf_{t\rightarrow\infty}\frac{N\left(D_{\alpha}(\Lambda^\beta(t)E_\beta(1)),1\right)}{g(t)}=\liminf_{t\rightarrow\infty}\frac{N\left(E^{1/\alpha}_\beta(1)D_{\alpha}(\Lambda^\beta(t)),1\right)}{g(t)}\\
		&=\liminf_{t\rightarrow\infty}\frac{N\left(E^{1/\alpha}_\beta(1)D_{\alpha}(\Lambda^\beta(t)),1\right)}{E^{1/\alpha}_\beta(1)D_{\alpha}(\Lambda^\beta(t))}\frac{E^{1/\alpha}_\beta(1)D_{\alpha}(\Lambda^\beta(t))}{g(t)}\\
		&=E^{1/\alpha}_\beta(1)\cdot \liminf_{t\rightarrow\infty}\frac{D_{\alpha}(\Lambda^\beta(t))}{g(t)},~{a.s.}~~(\text{using }\eqref{LIL-Poisson})\\
		&=E^{1/\alpha}_\beta(1)\liminf_{\Lambda(t)\rightarrow\infty}\frac{D_\alpha(\Lambda^\beta(t))}{g(t)},~{a.s.}\\
		&=\alpha\left(1-\alpha\right)^{(1-\alpha)/\alpha}E^{1/\alpha}_\beta(1)~{a.s.},
		\end{align*}
		where the last step follows from \eqref{LIL-sub}.
	\end{proof}
\end{theorem}	

It is easy to find the generating function of the NSTFPP from the generating function of the STFPP. It is given by
\begin{align}
G^{\alpha,\beta}_\Lambda(s,t) = L_\beta(- (1-s)^\alpha\Lambda^\beta(t)), ~~|s| \leq 1.	\end{align}
We now derive the governing equation for the {\it pgf} $G^{\alpha,\beta}_\Lambda(s,t)$ of the NSTFPP.
\begin{theorem}
	The {\it pgf} $G^{\alpha,\beta}_\Lambda(s,t)$ of the NSTFPP solves the following fractional differential equation
	\begin{equation}
	D_t^\beta G^{\alpha,\beta}_\Lambda(s,t)=\sum_{k=0}^{\infty}\frac{(-1)^{k+1}(1-s)^{\alpha(k+1)}}{\Gamma(\beta(k+1)+1)}D_t^\beta\Lambda^{\beta(k+1)}(t),
	\end{equation}\label{frac-gen}
	with the initial condition $G^{\alpha,\beta}_\Lambda(s,0)=0$.  $D^\beta_t(\cdot)$ denotes the Caputo fractional derivative of order $\beta$ defined in \eqref{cd}.
\end{theorem}
\begin{proof}
	We begin by directly computing the Caputo derivative of $G^{\alpha,\beta}_\Lambda(s,t)$
	\begin{align*}
	D_t^\beta G^{\alpha,\beta}_\Lambda(s,t)&=D_t^\beta L_\beta\left(-(1-s)^\alpha\Lambda^\beta(t)\right)\\
	&=\frac{1}{\Gamma(1-\beta)}\int_{0}^{t}\frac{\frac{d}{du}L_\beta\left(-(1-s)^\alpha\Lambda^\beta(u)\right)}{(t-u)^\beta}du\\
	&=\frac{1}{\Gamma(1-\beta)}\int_{0}^{t}\frac{-(1-s)^\alpha\lambda(u)\Lambda^{\beta-1}(u)L^1_{\beta,\beta}\left(-(1-s)^\alpha\Lambda^\beta(u)\right)}{(t-u)^\beta}du\\
	&=\frac{1}{\Gamma(1-\beta)}\int_{0}^{t}\frac{-(1-s)^\alpha\lambda(u)\Lambda^{\beta-1}(u)}{(t-u)^\beta}\sum_{k=0}^{\infty}\frac{(-1)^k(1-s)^{\alpha k}\Lambda^{\beta k}(u)}{\Gamma(\beta k+\beta)}du\\
	&=\frac{1}{\Gamma(1-\beta)}\sum_{k=0}^{\infty}\frac{(-1)^{k+1}(1-s)^{\alpha(k+1)}}{\Gamma(\beta k+\beta)}\int_{0}^{t}\frac{\lambda(u)\Lambda^{\beta(k+1)-1}(u)}{(t-u)^\beta}du\\
	&=\frac{1}{\Gamma(1-\beta)}\sum_{k=0}^{\infty}\frac{(-1)^{k+1}(1-s)^{\alpha(k+1)}}{\Gamma(\beta k+\beta)}\int_{0}^{t}\frac{\frac{d}{du}\Lambda^{\beta(k+1)}(u)}{(\beta k+\beta)(t-u)^\beta}du\\
	&=\sum_{k=0}^{\infty}\frac{(-1)^{k+1}(1-s)^{\alpha(k+1)}}{\Gamma(\beta k+\beta+1)}D_t^\beta \Lambda^{\beta(k+1)}(t),
	\end{align*}
	which completes the proof. 
\end{proof}
We next present an alternate representation of the generating function of the NSTFPP process $G^{\alpha,\beta}_\Lambda(s,t)$. This result is a generalization of the representation obtained by Orsingher and Polito (see \cite[Remark 2.4]{sfpp}) for the STFPP.
\begin{theorem} Let $U_i\stackrel{iid}{\sim}U[0,1]$, $i=1,\dots$, then 
	\begin{equation*}
	G^{\alpha,\beta}_\Lambda(s,t)= \mathbb{P} \left[ \min_{0\leq i \leq N^\alpha_\beta(t)} U_i^{1/\alpha}
	\geq 1-s \right]. \notag
	\end{equation*}
	\begin{proof}
		The generating function of the NSTFPP $G^{\alpha,\beta}_\Lambda(s,t)$ can be written as
		\begin{align*}
		G^{\alpha,\beta}_\beta(s,t) & = L_\beta(-\Lambda^\beta(t) (1-s)^\alpha) \\
		& = \sum_{k=0}^\infty (-1)^k \frac{(\Lambda^\beta(t))^k(1-s)^{\alpha k}}{\Gamma(\beta k+1)} \notag \\
		& = \sum_{k=0}^{\infty} (-1)^k \frac{(\Lambda^\beta(t))^k}{\Gamma(\beta k+1)}
		\sum_{i=0}^k (-1)^i \binom{k}{i} \left[1-(1-s)^\alpha\right]^i \notag \\
		& = \sum_{i=0}^\infty \left[ 1-(1-s)^\alpha \right]^i \sum_{k=i}^\infty (-1)^{k-i}
		\binom{k}{i} \frac{(\Lambda^\beta(t))^k}{\Gamma(\beta k+1)} \notag \\
		& = \sum_{i=0}^\infty \left[ \mathbb{P} \left( U_i^{1/\alpha} \geq 1-s \right) \right]^i
		\mathbb{P}\left[ N_\beta(t) = i \right] \notag \\
		& = \mathbb{P} \left[ \min_{0\leq i \leq N^\alpha_\beta(t)} U_i^{1/\alpha}
		\geq 1-s \right]. \qedhere
		\end{align*}
	\end{proof}
	
\end{theorem}
\subsection*{Arrival times of $\{W^\alpha_\beta(t)\}_{t\geq0}$} Define 
\begin{equation*}
	J_n=\min\{t\geq0:W^\alpha_\beta(t)=n\},
\end{equation*}
which denotes the arrival times of the NSTFPP $\{W^\alpha_\beta(t)\}_{t\geq0}$. Note that $\{J_n\leq t\}=\{W^\alpha_\beta(t)\geq n\}$, its distribution function is given by
\begin{align}
	F_{J_n}(t)&=\mathbb{P}[J_n\leq t]=\mathbb{P}[W^\alpha_\beta(t)\geq n]=\sum_{r=n}^{\infty}\mathbb{P}[W^\alpha_\beta(t)=r]\nonumber\\
	&=\sum_{r=n}^{\infty}\mathbb{P}\left[N\left(D_\alpha(E_\beta(\Lambda(t))),1\right)=r\right]\nonumber\\
&=\sum_{r=n}^{\infty}\mathbb{P}\left[N\left(\Lambda^{\beta/\alpha}(t)D_\alpha(E_\beta(1)),1\right)=r\right]=\sum_{r=n}^{\infty}\mathbb{P}\left[N\left(\Lambda^{\beta/\alpha}(t),D_\alpha(E_\beta(1))\right)=r\right]\nonumber\\
&=\sum_{r=n}^{\infty}\int_{0}^{\infty}\mathbb{P}\left[N\left(\Lambda^{\beta/\alpha}(t),x\right)=r\right]h^\alpha_\beta(x,1)dx,\nonumber\\
\shortintertext{using Fubini's theorem as the intergrand is positive}
&		=\int_{0}^{\infty}\sum_{r=n}^{\infty}\mathbb{P}\left[N\left(\Lambda^{\beta/\alpha}(t),x\right)=r\right]h^\alpha_\beta(x,1)dx\nonumber\\
&=\int_{0}^{\infty}\mathbb{P}\left[N\left(\Lambda^{\beta/\alpha}(t),x\right)\geq n\right]h^\alpha_\beta(x,1)dx,\nonumber
\end{align}
where $h^\alpha_\beta(x,t)$ is the {\it pdf} of $D_\alpha(E_\beta(t))$.

The above proved results holds valid for $\alpha=1$ which leads to the the non-homogeneous time fractional Poisson process (NTFPP) and also for $\beta=1$ which leads to the non-homogeneous space fractional Poisson process (NSFPP). It is known (see \cite{sfpp}) that the mean of the SFPP process is infinite and hence it is not possible to investigate properties involving mean and covariance. On the other hand, we can study the results for the TFPP process.  We therefore present the definition and prove some results for the NTFPP process. We henceforth discuss only the NTFPP process
\subsection{Non-homogeneous time fractional Poisson process}
	\begin{definition}[NTFPP]\label{def:NTFPP}
		The non-homogeneous time fractional Poisson process (NTFPP) is defined as 
		\begin{equation}\label{def:ntfpp}
		W^{(1)}_\beta(t)=N_\beta(\Lambda(t),1),~t\geq0,
		\end{equation}
		where $\{N_\beta(t,\lambda)\}_{t\geq0}$ is the TFPP and $\Lambda(t)=\int_{0}^{t}\lambda(u)du$ is the rate function with intensity function $\lambda(u),u\geq0.$
	\end{definition}

	\noindent  The mean,  variance (see \cite{lask,beghinejp2009}) and the covariance functions (see \cite[eq. (14)]{LRD2014}) of the TFPP are given by 
	\begin{align}
	\mathbb{E}[N_{\beta}(t,\lambda)] &= qt^{\beta};~\label{fppmean} 
	\mbox{Var}[N_{\beta}(t,\lambda)]=q t^{\beta}+Rt^{2\beta}, \\
	\text{Cov}[N_{\beta}(s,\lambda),N_{\beta}(t,\lambda)]&=qs^{\beta}+ ds^{2\beta}+ q^{2}[\beta t^{2\beta}B(\beta,1+\beta;s/t)-(st)^{\beta}],\label{fpp-cov}
	\end{align}
	\noindent
	$0<s\leq t$, where $q=\lambda/\Gamma(1+\beta)$, $R=\frac{\lambda ^{2}}{\beta}\left(\frac{1}{\Gamma(2\beta)}-\frac{1}{\beta\Gamma^{2}(\beta)}\right)>0$,  $d=\beta q^{2}B(\beta, 1+\beta)$, and $B(a,b;x)=\int_{0}^{x}t^{a-1}(1-t)^{b-1}dt,~0<x<1$, is the incomplete beta function.
\noindent We next give the mean, the variance and the covariance of the NTFPP.
\begin{theorem}\label{mean-var-cov-tcfpp-1}
	Let $0<s\leq t<\infty$, $q_1=1/\Gamma(1+\beta)$ and $d_1=\beta q_1^{2}B(\beta, 1+\beta)$. The distributional properties of the NTFPP $\{W^{(1)}_\beta(t)\}_{t\geq0}$ are as follows:\vspace*{0.25cm}\\
	(i) $~~\mathbb{E}[W^{(1)}_\beta(t)]=q_1\Lambda^\beta(t)$,\\
	(ii)$~~\text{Var}[W^{(1)}_\beta(t)]=q_1\Lambda^{\beta}(t)(1-q_1\Lambda^{\beta}(t))+2d_1\Lambda^{2\beta}(t)$,
	\begin{flalign}
	\text{(iii) }~\text{Cov}[W^{(1)}_\beta(s),W^{(1)}_\beta(t)]&=q_1\Lambda^{\beta}(s)+ d_1\Lambda^{2\beta}(s)-q_1^{2}\Lambda^{\beta}(s)\Lambda^{\beta}(t) &&\nonumber&\\&~~~
	+ q_1^{2}\beta  \Lambda^{2\beta}(t)B(\beta,1+\beta;\Lambda(s)/\Lambda(t)),~~~\text{for }0<s\leq t.\nonumber&
	\end{flalign}
\end{theorem}
\begin{proof}
	Note that using \eqref{fppmean},
	\begin{align}\label{tcfpp-mean}
	\mathbb{E}[W^{(1)}_\beta(t)]&=\mathbb{E}[N_{\beta}(\Lambda(t),1)]=q_1\Lambda^\beta(t),
	\end{align}
	which proves Part (i). From \eqref{fpp-cov} and \eqref{fppmean}, 
	\begin{equation}\label{efppsfppt}
	\mathbb{E}[N_{\beta}(s,\lambda)N_{\beta}(t,\lambda)]=qs^{\beta}+d_1s^{2\beta}+ q^{2}\beta\left[t^{2\beta}B(\beta,1+\beta;s/t)\right]
	\end{equation}
	which leads to
	\begin{align}
	\mathbb{E}[W^{(1)}_\beta(s)W^{(1)}_\beta(t)]&=\mathbb{E}\left[N_{\beta}(\Lambda(s),1)N_{\beta}(\Lambda(t),1)\right]\nonumber\\
	&=q_1\Lambda^{\beta}(s)+ d_1\Lambda^{2\beta}(s)+q_1^{2} \beta\Lambda^{2\beta}(t)B(\beta,1+\beta;\Lambda(s)/\Lambda(t)).\label{bivariate-fnbfp}
	\end{align}
	By  \eqref{tcfpp-mean} and \eqref{bivariate-fnbfp}, Part (iii) follows. Part (ii) follows from Part (iii) by putting $s=t$.
\end{proof}
\subsection*{Long-range dependence}

 We next prove the LRD property for the NTFPP $\{W_\beta^{(1)}(t)\}_{t\geq0}$ process with  
 Weibull rate function. First, we have the following result for the asymptotic expansion of the correlation function of the NTFPP. We first have the following definition.
 
 \begin{definition}
 	Let $f(x)$ and $g(x)$ be positive functions. We say that $f(x)$ is asymptotically equal to $g(x)$, written as $f(x)\sim g(x)$, as $x$ tends to infinity, if 
 	\begin{equation*}
 	\lim\limits_{x\rightarrow\infty}\frac{f(x)}{g(x)}=1.
 	\end{equation*}
 \end{definition}
 \begin{theorem} Let $\Lambda(t)\to\infty, $ as $t\to\infty$.
 	The correlation function of the NTFPP $\{W^{(1)}_\beta(t)\}_{t\geq 0}$ has the asymptotic expansion as,
 	\begin{equation}\label{asym-ntfpp}
 	\text{Corr}[W^{(1)}_\beta(s),W^{(1)}_\beta(t)]	\sim\Lambda^{-\beta}(t)\left(\frac{q_1\Lambda^{\beta}(s)+d_1\Lambda^{2\beta}(s)}{\sqrt{(2d_1-q_1^2)\text{Var}[W^{(1)}_\beta(s)]}}\right).
 	\end{equation}
 	\begin{proof}\noindent Consider the last term of $\text{Cov}[W^{(1)}_\beta(s),W^{(1)}_\beta(t)]$ given in Theorem \ref{mean-var-cov-tcfpp-1} (iii), namely,
 		$\beta q_1^{2} \Lambda^{2\beta}(t)B(\beta,1+\beta;\Lambda(s)/\Lambda(t)).$
 		We get the asymptotic expansion, for large $t$,
 		\begin{align}
 		\beta q_1^{2} \Lambda^{2\beta}(t)B(\beta,1+\beta;\Lambda(s)/\Lambda(t))&=\beta q_1^2\Lambda^{2\beta}(t)\int_{0}^{\frac{\Lambda(s)}{\Lambda(t)}}u^{\beta-1}(1-u)^\beta du\nonumber\\
 		&=\beta q_1^2\Lambda^{2\beta}(t)\left(\frac{1}{\beta}\left(\frac{\Lambda(s)}{\Lambda(t)}\right)^\beta-\frac{\beta}{1+\beta}\left(\frac{\Lambda(s)}{\Lambda(t)}\right)^{\beta+1}+O\left(\left(\frac{\Lambda(s)}{\Lambda(t)}\right)^{\beta+2}\right)\right)\nonumber\\
 		&\sim q_1^{2}\Lambda^{\beta}(s)\Lambda^{\beta}(t)\label{autocovariance-last-summand-1}.
 		\end{align}
 		\noindent Using \eqref{autocovariance-last-summand-1}, Theorem \ref{mean-var-cov-tcfpp-1} (iii) becomes for large $t$,
 		\begin{align}
 		\text{Cov}[W^{(1)}_\beta(s),W^{(1)}_\beta(t)]&\sim q_1\Lambda^{\beta}(s)+d_1\Lambda^{2\beta}(s).\label{covariance-large-t}
 		\end{align}
 		
 		\noindent Similarly, from Theorem \ref{mean-var-cov-tcfpp-1} (ii), we have that
 		\begin{align}
 		\text{Var}[W^{(1)}_\beta(t)]&=(2d_1-q_1^2)\Lambda^{2\beta}(t)+q_1\Lambda^\beta(t)\nonumber\\
 		&\sim(2d_1-q_1^2)\Lambda^{2\beta}(t).\label{variance-large-t}
 		\end{align}
 		
 		\noindent Thus, from \eqref{covariance-large-t} and \eqref{variance-large-t}, the correlation between $W^{(1)}_\beta(s)$ and $W^{(1)}_\beta(t)$ for large $t>s$, is
 		\begin{align}
 		\text{Corr}[W^{(1)}_\beta(s),W^{(1)}_\beta(t)]&=\frac{\text{Cov}[W^{(1)}_\beta(s),W^{(1)}_\beta(t)]}{\sqrt{\text{Var}[W_\beta^{(1)}(s)]\text{Var}[W_\beta^{(1)}(t)]}}\nonumber\\
 		&\sim\frac{q_1\Lambda^{\beta}(s)+d_1\Lambda^{2\beta}(s)}{\sqrt{\Lambda^{2\beta}(t)(2d_1-q_1^2)}\sqrt{\text{Var}[W^{(1)}_\beta(s)]}}= \Lambda^{-\beta}(t)\left(\frac{q_1\Lambda^{\beta}(s)+d_1\Lambda^{2\beta}(s)}{\sqrt{(2d_1-q_1^2)\text{Var}[W^{(1)}_\beta(s)]}}\right),\nonumber
 		\end{align}
 		\noindent which proves the result.
 	\end{proof}
 \end{theorem}
 \begin{corollary}
 	For the Weibull rate function (see Table \ref{table-example}), the NTFPP exhibits the LRD property for $0<a\beta<1$. To see this, we have that from \eqref{asym-ntfpp}, 
 	\begin{equation}
 	\text{Corr}[W^{(1)}_\beta(s),W^{(1)}_\beta(t)]	\sim t^{-a\beta}\left(\frac{q_1s^{a\beta}+d_1 s^{2a\beta}/b^{a\beta}}{\sqrt{(2d_1-q_1^2)\text{Var}[W^{(1)}_\beta(s)]}}\right).
 	\end{equation}
 	Using Definition \ref{LRD-definition}, we have the LRD property of the NTFPP for $0<a\beta<1$.
 \end{corollary}

\subsection{Fractional non-homogeneous Poisson process}
The fractional non-homogeneous Poisson process (FNPP) is introduced by Leonenko et. al. (see \cite{fnhpp}). They studied the governing fractional differential-integral-difference equation, distributional properties and arrival times. Here, we present some additional results related to the FNPP process. 
	\begin{definition}
		The fractional non-homogeneous Poisson process (FNPP), introduced by Leonenko et al. (see \cite{fnhpp}), is defined by time-changing the NPP by the inverse $\beta$-stable subordinator, that is,
		\begin{equation}
		W^{(2)}_\beta(t)=N(\Lambda(E_\beta(t)),1),~t\geq0.
		\end{equation}
	\end{definition}
\noindent We next present the limit theorem for the FNPP for the Weibull rate function.
\begin{theorem}
For Weibull rate function $\Lambda(t)=(t/b)^a,~a,b>0$, then 
\begin{equation}\label{fnpp-weibull}
\lim\limits_{t\to\infty}\frac{W^{(2)}_\beta(t)}{t^{a\beta}}=\left(\frac{E_\beta(1)}{b}\right)^a,~~a.s.
\end{equation}
\begin{proof}
We begin with
\begin{align*}
\lim\limits_{t\to\infty}\frac{W^{(2)}_\beta(t)}{t^{a\beta}}&=\lim\limits_{t\to\infty}\frac{N\left(E_\beta(t)^{a}/b^a,1\right)}{t^{a\beta}}=\lim\limits_{t\to\infty}\frac{N\left(t^{a\beta}E^a_\beta(1)/b^a,1\right)}{t^{a\beta}}\\
&=\lim\limits_{t\to\infty}\frac{N\left(t^{a\beta}E^a_\beta(1)/b^a,1\right)}{t^{a\beta}E^a_\beta(1)/b^a}\frac{t^{a\beta}E^a_\beta(1)/b^a}{t^{a\beta}}\\
&=\left(\frac{E_\beta(1)}{b}\right)^a,~~a.s., 
\end{align*}
using \eqref{LIL-Poisson}. 
\end{proof}
\end{theorem}
\noindent
We next show that the FNPP $\{W_\beta^{(2)}(t)\}_{t\geq0}$ with Weibull rate, exhibits the LRD property.
\begin{theorem}
The FNPP $\{W_\beta^{(2)}(t)\}_{t\geq0}$ with Weibull rate unction $\Lambda(t)=(t/b)^a,~a,b>0,$ exhibits the LRD property, when $0<a\beta<1.$
\begin{proof}
To investigate the LRD property, we study the asymptotic behavior of the correlation function of the FNPP. Note that the covariance function of the FNPP is given by (see \cite[Proposition 2]{fnhpp}) 
\begin{equation}
\text{Cov}[W_\beta^{(2)}(s),W_\beta^{(2)}(t)]=\mathbb{E}[\Lambda(E_\beta(s))]+\text{Cov}[\Lambda(E_\beta(s)),\Lambda(E_\beta(t))],~s<t.
\end{equation}
We first compute second part of the above equation for the Weibull rate function
\begin{align}
\text{Cov}\left[\left(E_\beta(s)/b\right)^a,\left(E_\beta(t)/b\right)^a\right]&=\frac{1}{b^{2a}}\left[\mathbb{E}\left[E^a_\beta(s)E_\beta^a(t)\right]-\mathbb{E}\left[E^a_\beta(s)\right]\mathbb{E}\left[E_\beta^a(t)\right]\right]\nonumber\\
&=\frac{1}{b^{2a}}\left[\mathbb{E}\left[(st)^{a\beta}E^a_\beta(1)E_\beta^a(1)\right]-\mathbb{E}\left[s^{a\beta}E^a_\beta(1)\right]\mathbb{E}\left[t^{a\beta}E_\beta^a(1)\right]\right]~(\text{using }\eqref{istable-ss})\nonumber\\
&=\frac{(st)^{a\beta}}{b^{2a}}\left[\mathbb{E}\left[E_\beta^{2a}(1)\right]-\left(\mathbb{E}\left[E^a_\beta(1)\right]\right)^2\right]\nonumber\\
&=\frac{(st)^{a\beta}}{b^{2a}}\text{Var}\left[E_\beta^{a}(1)\right].\label{secondpart-LRD}
\end{align}
It follows that 
\begin{align}
\text{Cov}[W_\beta^{(2)}(s),W_\beta^{(2)}(t)]&=\mathbb{E}\left[\left(E_\beta(s)/b\right)^a\right]+\frac{(st)^{a\beta}}{b^{2a}}\text{Var}\left[E_\beta^{a}(1)\right]\nonumber\\
&=\frac{s^{a\beta}}{b^a}\mathbb{E}\left[E^a_\beta(1)\right]+\frac{(st)^{a\beta}}{b^{2a}}\text{Var}\left[E_\beta^{a}(1)\right].\label{cov-fnpp}
\end{align}
Note that the variance of the FNPP (see \cite[eq. (4.5)]{fnhpp}) is given by
\begin{equation*}
\text{Var}[W^{(2)}_\beta(t)]=\mathbb{E}[\Lambda(E_\beta(t))]+\text{Var}[\Lambda(E_\beta(t))].
\end{equation*}
 The variance of the FNPP, in case of the Weibull rate function,  reduces to
\begin{align}
\text{Var}[W^{(2)}_\beta(t)]&=\mathbb{E}\left[\left(E_\beta(t)/b\right)^a\right]+\text{Var}\left[\left(E_\beta(t)/b\right)^a\right]\nonumber\\
&=\frac{t^{a\beta}}{b^a}\mathbb{E}\left[E^a_\beta(1)\right]+\frac{t^{2a\beta}}{b^{2a}}\text{Var}\left[E^a_\beta(1)\right]\nonumber\\
&\sim \frac{t^{2a\beta}}{b^{2a}}\text{Var}\left[E^a_\beta(1)\right],~\text{for large }t.\label{var-asym}
\end{align}
Thus, from \eqref{cov-fnpp} and \eqref{var-asym}, the correlation function between $W^{(2)}_\beta(s)$ and $W_\beta^{(2)}(t)$ for large $t>s$, is
\begin{align*}
\text{Corr}[W^{(2)}_\beta(s),W^{(2)}_\beta(t)]&=\frac{\text{Cov}[W^{(2)}_\beta(s),W^{(2)}_\beta(t)]}{\sqrt{\text{Var}[W_\beta^{(2)}(s)]\text{Var}[W_\beta^{(2)}(t)]}}\nonumber\\
&\sim \frac{\frac{s^{a\beta}}{b^a}\mathbb{E}\left[E^a_\beta(1)\right]+\frac{(st)^{a\beta}}{b^{2a}}\text{Var}\left[E_\beta^{a}(1)\right]}{\frac{t^{a\beta}}{b^{a}}\sqrt{\text{Var}[W_\beta^{(2)}(s)]}\sqrt{\text{Var}\left[E^a_\beta(1)\right]}}\\
&= t^{-a\beta}\frac{s^{a\beta}\mathbb{E}\left[E^a_\beta(1)\right]}{\sqrt{\text{Var}[W_\beta^{(2)}(s)]\text{Var}\left[E^a_\beta(1)\right]}}+\dfrac{s^{a\beta}}{b^a}\sqrt{\frac{\text{Var}\left[E_\beta^{a}(1)\right]}{\text{Var}[W_\beta^{(2)}(s)]}}\\
&=d_1(s)t^{-a\beta}+d_2(s),~\text{(say)}.
\end{align*}
Hence, the correlation function $\text{Corr}[W^{(2)}_\beta(s),W^{(2)}_\beta(t)]$ behaves like $t^{-a\beta}d_1(s)+d_2(s)$, for $0<a\beta<1$ and so the FNPP $\{W^{(2)}_\beta(t)\}_{t\geq0}$ exhibits the LRD property.
\end{proof}
\end{theorem}
\section{Simulation}\label{sec:Simulation}
\noindent In this section we present  simulated sample paths for some NSTFPP. 

The sample paths of the NTFPP, the NSFPP and the NSTFPP processes are simulated for the Makeham's distribution with cumulative hazard function $\Lambda(t)=(a/b)e^{bt}-(a/b)+\mu t,a,b,\mu>0$ as given in Table \ref{table-example}.
\begin{figure}[!ht]
	\begin{subfigure}{0.5\textwidth}
		\caption{Sample paths of the NTFPP process for $\beta=0.5,a=0.6,b=0.1,\mu=5.0$ and $T=50.$}
		\includegraphics[width=1\textwidth]{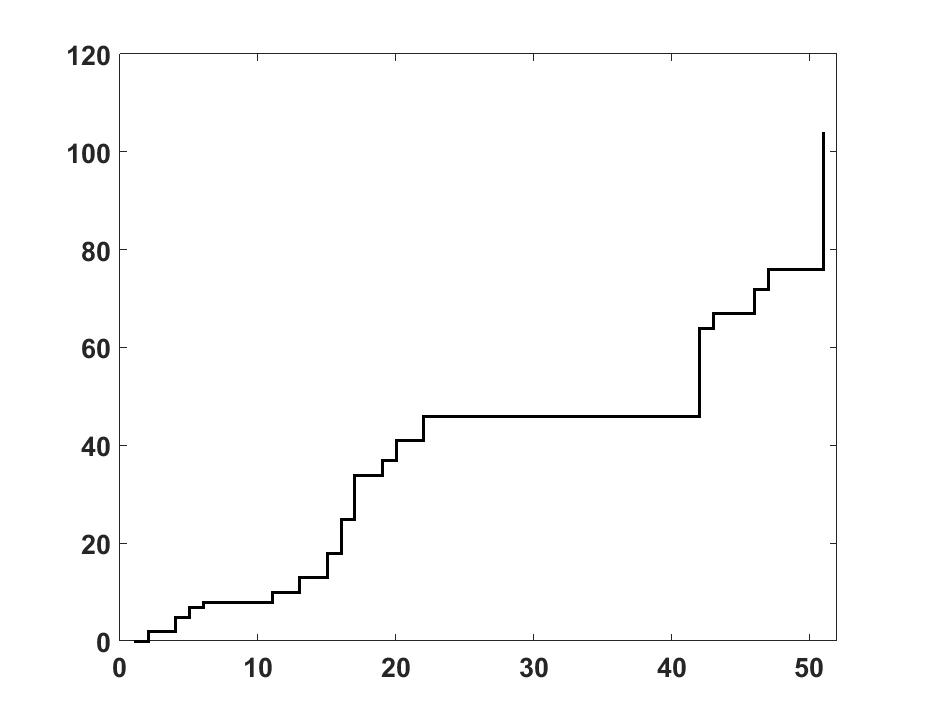}
	\end{subfigure}~
	\begin{subfigure}{0.5\textwidth}
		\caption{Sample paths of the NTFPP process for $\beta=0.9,a=0.6,b=0.1,\mu=5.0$ and $T=50.$}
		\includegraphics[width=1\textwidth]{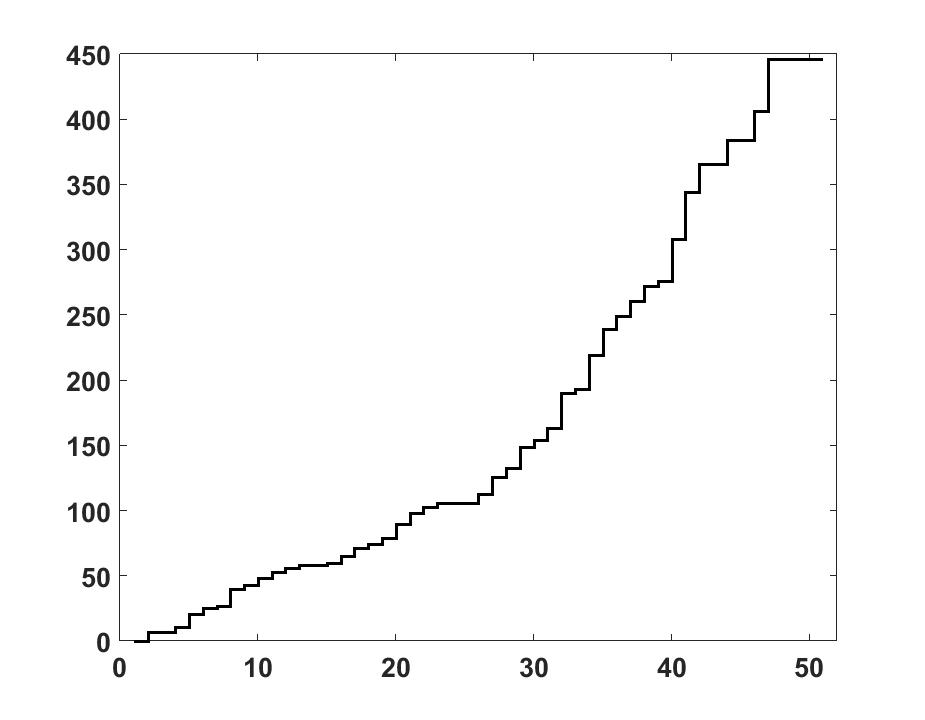}
	\end{subfigure}
	\caption{Sample paths of the NTFPP process}
\end{figure}~
\begin{figure}[H]
	\begin{subfigure}{0.5\textwidth}
		\caption{Sample paths of the NSFPP process for $\alpha=0.5,a=0.6,b=0.1,\mu=5.0$ and $T=50.$}
		\includegraphics[width=1\textwidth]{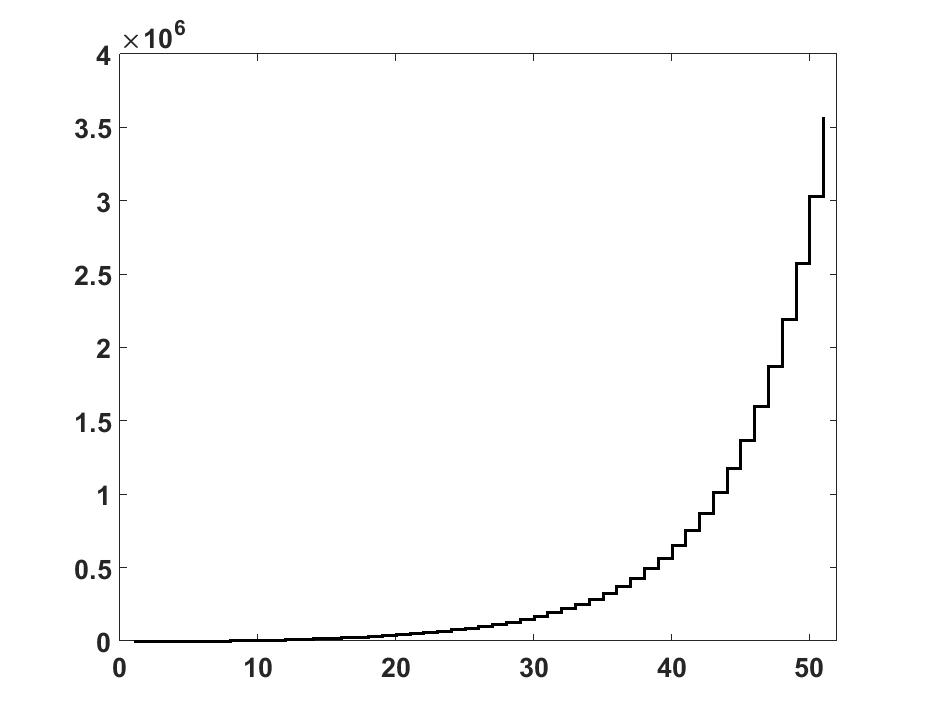}
	\end{subfigure}~
	\begin{subfigure}{0.5\textwidth}
		\caption{Sample paths of the NSFPP process for $\alpha=0.9,a=0.6,b=0.1,\mu=5.0$ and $T=50.$}
		\includegraphics[width=1\textwidth]{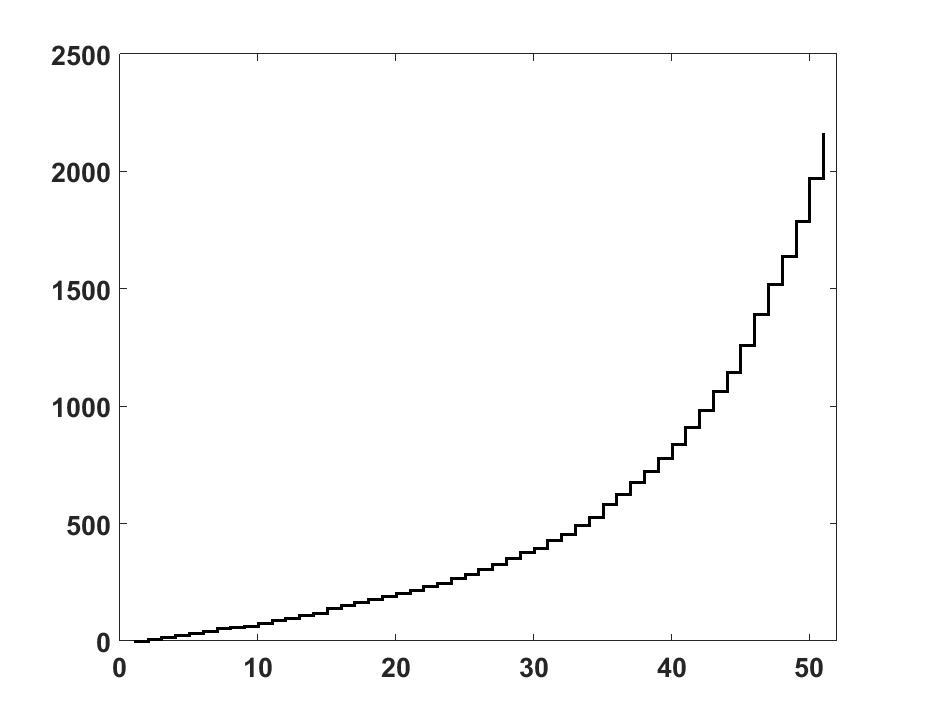}
	\end{subfigure}
	\caption{Sample paths of the NSFPP process}
\end{figure}

\begin{figure}[H]
	\begin{subfigure}{0.5\textwidth}
		\caption{Sample paths of the NSTFPP process for $\alpha=0.8,\beta=0.6,a=0.6,b=0.1,\mu=5.0$ and $T=50.$}
		\includegraphics[width=1\textwidth]{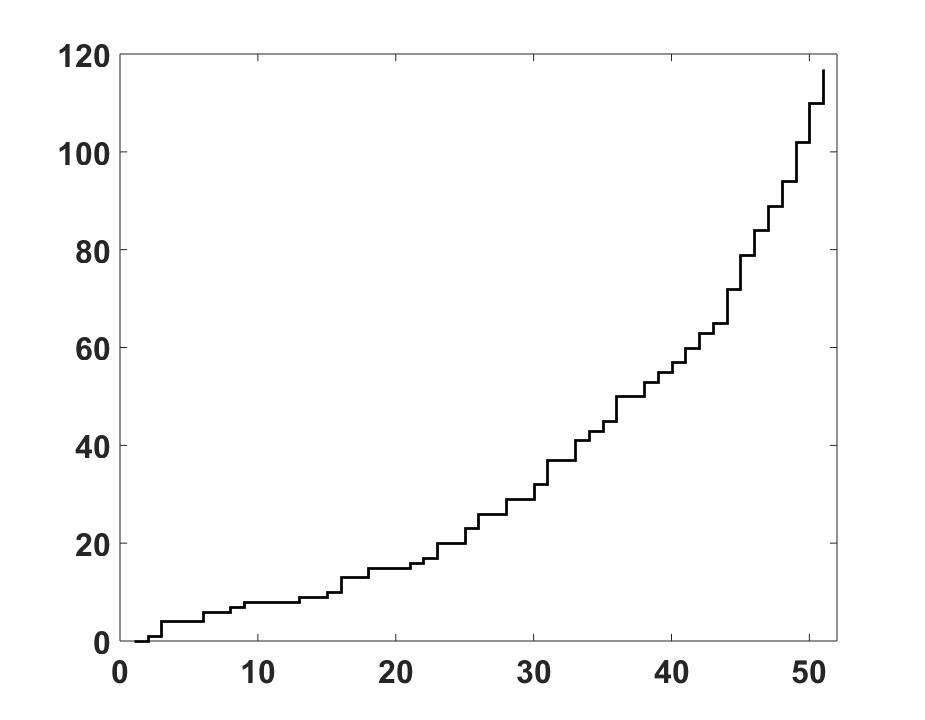}
	\end{subfigure}~
	\begin{subfigure}{0.5\textwidth}
		\caption{Sample paths of the NSTFPP process for $\alpha=0.5,\beta=0.9,a=0.6,b=0.1,\mu=5.0$ and $T=50.$}
		\includegraphics[width=1\textwidth]{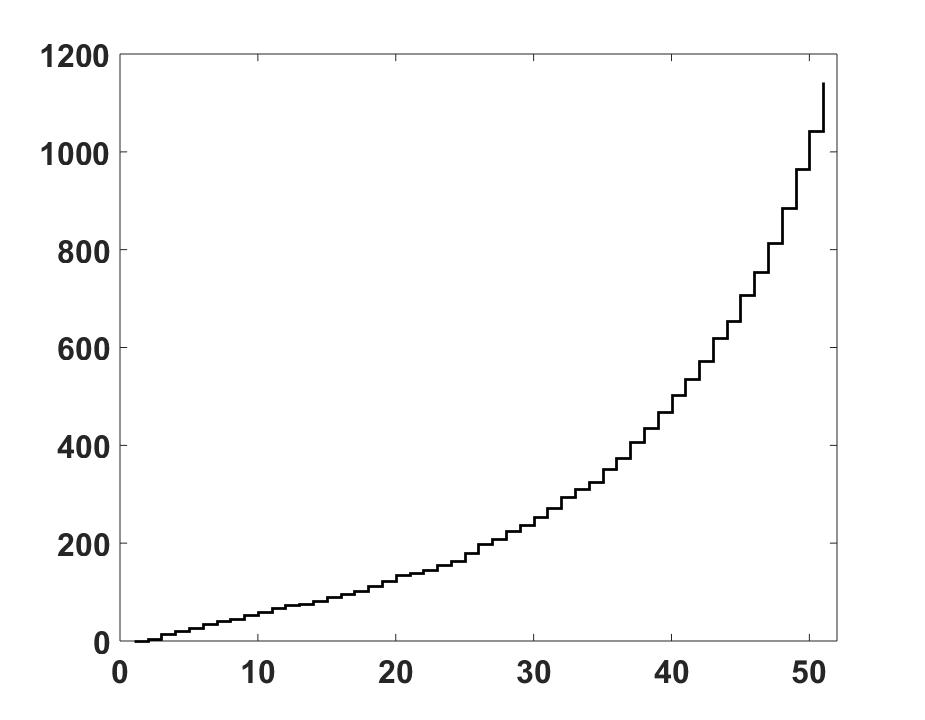}
	\end{subfigure}
	\caption{Sample paths of the NSTFPP process}
\end{figure}

\bibliographystyle{abbrv}
\bibliography{researchbib}

\end{document}